\documentclass[12pt,reqno]{amsart}
\usepackage{amssymb, amstext, amsmath, amsthm}
\usepackage{faktor}
\usepackage[all]{xy}
\usepackage{mathtools}
\usepackage{nicefrac}
\usepackage{hyperref}
\usepackage{xfrac}
\numberwithin{equation}{section}
\theoremstyle{plain}
\newtheorem{thm}{Theorem}[section]

\newtheorem{theorem}[thm]{Theorem}
\newtheorem{proposition}[thm]{Proposition}
\newtheorem{lemma}[thm]{Lemma}
\newtheorem{corollary}[thm]{Corollary}

\theoremstyle{definition}

\newtheorem{remark}[thm]{Remark}

\mathtoolsset{centercolon}
\newcommand{\bC}{{\mathbb{C}}}

\newcommand{\bZ}{{\mathbb{Z}}}

  \newcommand{\D}{{\mathcal{D}}}

  \newcommand{\K}{{\mathcal{K}}}  
\renewcommand{\L}{{\mathcal{L}}}
  \newcommand{\M}{{\mathcal{M}}}
  
\renewcommand{\O}{{\mathcal{O}}}

\newcommand{\fg}{{\mathfrak{g}}}
\newcommand{\fH}{{\mathfrak{H}}}
\newcommand{\fh}{{\mathfrak{h}}}

\def\ii{\mathrm{i}}

\def\bt{{\boldsymbol{t}}}
\def\bm{{\boldsymbol{m}}}
\def\bn{{\boldsymbol{n}}}
\def\bp{{\boldsymbol{p}}}
\def\bq{{\boldsymbol{q}}}
\def\bd{{\boldsymbol{\delta}}}

\renewcommand{\phi}{\varphi}

\newcommand{\upchi}{{\raise.35ex\hbox{$\chi$}}}

\def\vac{\boldsymbol{1}}

\DeclareMathOperator\Aut{Aut}
\DeclareMathOperator\End{End}
\DeclareMathOperator\Ind{Ind}
\DeclareMathOperator\Span{span}

\allowdisplaybreaks

\begin{document}

\title[Representations of Twisted Toroidal Lie Algebras]{Representations of Twisted Toroidal Lie Algebras from Twisted Modules over Vertex Algebras}
\author{Bojko Bakalov}
\author{Samantha Kirk}
\address{Department of Mathematics, North Carolina State University, Raleigh, NC 27695, USA}
\email{bojko\_bakalov@ncsu.edu}
\email{slkirk@ncsu.edu}
\date{September 2, 2020. Revised February 14, 2021}
\maketitle
\begin{abstract}
Given a simple finite-dimensional Lie algebra and an automorphism of finite order, one defines the notion of a twisted toroidal Lie algebra. 
In this paper, we construct representations of twisted toroidal Lie algebras from twisted modules over affine and lattice vertex algebras. 
\end{abstract}
\section{Introduction}\label{sec:1}
A toroidal Lie algebra is the universal central extension of the tensor product of a simple finite-dimensional Lie algebra $\fg$ and the ring of Laurent polynomials in $r$ variables. Moody, Rao, and Yokonuma were the first to introduce toroidal Lie algebras in \cite{mry}, which led to constructions of representations of toroidal Lie algebras from vertex operators in works such as \cite{yb, em, fab, jing, jing2, tan1, tan2}. In a fashion similar to twisted affine Lie algebras \cite{IDLA}, Fu and Jiang \cite{FJ} used diagram automorphisms of $\fg$ of finite order to construct twisted toroidal Lie algebras and studied their integrable modules.  Some representations of twisted toroidal Lie algebras were constructed in works such as \cite{abp,cjkt,misra}.

Vertex algebras have proven to be a useful tool in constructing vertex operator representations of infinite-dimensional Lie algebras \cite{b, curves, flm, VAFB, bombay, li}.  Some examples include the use of lattice vertex algebras to generalize the Frenkel--Kac realization of affine Kac--Moody algebras in terms of vertex operators \cite{b, fk} and the use of tensor products of vertex algebras to create representations of untwisted toroidal Lie algebras by Berman, Billig, and Szmigielski \cite{bm}. 
Twisted vertex operators first appeared in applications such as the principal realization of affine Kac--Moody algebras \cite{KKLW, lep3} and the Frenkel--Lepowsky--Meurman construction of the moonshine vertex operator algebra \cite{flm, gannon}. These twisted vertex operators led to the notion of a twisted module over a vertex algebra \cite{bak, dong, gold, lep1}. 

In this paper, we use twisted modules over affine and lattice vertex algebras to construct representations of twisted toroidal Lie algebras. Our paper is influenced by Berman, Billig, and Szmigielski's work on constructing representations of untwisted toroidal Lie algebras in \cite{bm, sbm, yb} and we will rely on their notation in our description of toroidal Lie algebras. 

The paper is organized as follows.
In Section \ref{sec:2}, we provide background information on vertex algebras. 
In Section \ref{sec:voraff}, we connect representations of affine Kac--Moody Lie algebras with twisted modules over affine vertex algebras.  
In Section \ref{sec:vortor}, we use the representations of affine Kac--Moody algebras presented in Section \ref{sec:voraff} as the framework to build representations of twisted toroidal Lie algebras from twisted modules over a tensor product of an affine vertex algebra and lattice vertex algebras. 
Unless otherwise specified, all vector spaces, linear maps and tensor products will be over the field $\bC$ of complex numbers.

\section{Preliminaries}\label{sec:2}

The purpose of this section is to review basic definitions and establish notation for vertex algebras. 
For more details, we refer the reader to \cite{bak, curves, VAFB, li}.

\subsection{Vertex algebras}\label{sec:va}

Recall that a \emph{vertex algebra} \cite{curves, VAFB, li}
is a vector space $V$ with a distinguished vector $\vac\in V$ 
(\emph{vacuum vector}), equipped with bilinear \emph{$n$-th products} for $n\in\bZ$:
\begin{equation}\label{vert1}
V\otimes V\to V, \qquad a\otimes b \mapsto a_{(n)}b,
\end{equation}
subject to the following axioms. First, for every fixed $a,b\in V$, we have
$a_{(n)}b = 0$ for sufficiently large $n$ (denoted $n\gg0$).
Second, the vacuum vector $\vac$ plays the role of a unit in the sense that
\begin{equation*}
a_{(-1)}\vac = \vac_{(-1)} a = a, \qquad 
a_{(n)}\vac = 0, \quad n\geq 0.
\end{equation*}
Finally, the main axiom of a vertex algebra is the \emph{Borcherds identity}
(also called Jacobi identity \cite{flm})
satisfied for all $a,b,c \in V$ and $k,m,n\in\bZ$:
\begin{equation}\label{vert5}
\begin{split}
\sum_{j=0}^\infty & \binom{m}{j} (a_{(k+j)}b)_{(m+n-j)}c
= \sum_{i=0}^\infty \binom{k}{i} (-1)^i
a_{(m+k-i)}(b_{(n+i)}c)
\\
&- \sum_{i=0}^\infty \binom{k}{i} (-1)^{k+i} \, b_{(n+k-i)}(a_{(m+i)}c).
\end{split}
\end{equation}

Note that all sums in \eqref{vert5} are in fact finite. We can view \eqref{vert1} as defining
a sequence of linear operators $a_{(n)}$ on $V$, for $a\in V$, $n\in\bZ$, called the
\emph{modes} of $a$.
Setting $k=0$ in the
Borcherds identity, we obtain the \emph{commutator formula} 
\begin{equation}\label{vert6}
[a_{(m)}, b_{(n)}]=\sum_{j=0}^\infty \binom{m}{j} (a_{(j)}b)_{(m+n-j)},
\end{equation}
which will be very useful for us.

It is convenient to organize the modes into formal power series
\begin{equation}\label{vert7}
Y(a,z)=\sum_{n\in \mathbb{Z} }a_{(n)} z^{-n-1}, \qquad a\in V,
\end{equation}
called \emph{fields} or vertex operators.
The linear map $Y\colon V\to(\End V)[[z,z^{-1}]]$ is known as the \emph{state-field correspondence}.
Observe that $Y(\vac,z)=I$ is the identity operator.

Introduce the \emph{translation operator} $T\in \End(V)$ defined by $Ta=a_{(-2)}\vac$.
Then
\begin{equation}\label{vert2}
[T, Y(a,z)]=Y(Ta,z) = \partial_zY(a,z),
\end{equation}
or equivalently,
\begin{equation}\label{vert2b}
[T,a_{(n)}]=(Ta)_{(n)}=-n a_{(n-1)}.
\end{equation}
In particular, $T$ is a derivation of all $n$-th products.

Another important consequence of the Borcherds identity is the \emph{$(-1)$-st product identity} $(a,b\in V)$:
\begin{equation}\label{vert3}
\begin{split}
Y(a_{(-1)}b,z) &={:}Y(a,z)Y(b,z){:}
\\
&=\sum_{n<0} a_{(n)} z^{-n-1} Y(b,z) 
+\sum_{n\ge0} Y(b,z) a_{(n)} z^{-n-1}.
\end{split}
\end{equation}
The double colons in \eqref{vert3} denote the so-called \emph{normally-ordered product}. Combining \eqref{vert2}--\eqref{vert3}, we get
\begin{equation}\label{vert4}
Y(a_{(-1-m)}b,z) = \frac1{m!} {:}(\partial_z^m Y(a,z))Y(b,z){:},
\qquad m\ge0.
\end{equation}

Finally, recall that the tensor product of two vertex algebras $V_1$ and $V_2$ is again a vertex algebra \cite{FHL} with a vacuum vector $\vac\otimes\vac$ and a state-field correspondence given by
($a\in V_1$, $b\in V_2$):
\begin{equation}\label{vert9}
Y(a\otimes b, z)=Y(a, z)\otimes Y(b,z)=\sum_{k,m\in \bZ}a_{(k)}\otimes b_{(m)}z^{-k-m-2}.
\end{equation}
In terms of modes, we have
\begin{equation}\label{vert10}
(a\otimes b)_{(n)}=\sum_{k\in \bZ}a_{(k)}\otimes b_{(n-k-1)}.
\end{equation}
The translation operator in $V_1\otimes V_2$ is $T\otimes I + I\otimes T$.

\subsection{Twisted modules over vertex algebras}\label{sec:twva}
Let $V$ be a vertex algebra and $\sigma$ be an automorphism of $V$, i.e., a linear operator such that $\sigma(\vac)=\vac$ and $\sigma(a_{(n)}b) = (\sigma a)_{(n)}(\sigma b)$ for all
$a,b\in V$, $n\in\bZ$. Suppose that $\sigma^N=I$ for some positive integer $N$. Then $\sigma$ is diagonalizable on $V$.

A \emph{$\sigma$-twisted $V$-module} \cite{dong, gold} is a vector space $M$ endowed with a linear map
$Y^M\colon V\to (\End M)[[z^{1/N},z^{-1/N}]]$,
\begin{equation}\label{vert11}
Y^M(a,z)=\sum_{m\in \frac1N\bZ } a_{(m)}^M z^{-m-1}, \qquad a\in V,
\end{equation}
subject to the following axioms. First, for every $a\in V$, $v\in M$, we have $a_{(m)}^M v=0$ for $m\gg0$. Next, $Y^M(\vac,z)=I$ and 
\begin{equation}\label{vert12}
Y^M(\sigma a , z)=Y^M(a, e^{2\pi\ii}z),
\end{equation}
where the meaning of the right-hand side is that we replace $z^{-m-1}$ with $e^{-2\pi\ii(m+1)} z^{-m-1}$ in each summand of \eqref{vert11}.
Explicitly, \eqref{vert12} means that if $a$ is an eigenvector of $\sigma$, then in \eqref{vert11} we only have terms with $m\in \frac1N\bZ$ such that
$\sigma a = e^{-2\pi\ii m} a$.
Finally, we have the twisted \emph{Borcherds identity} for any $a,b\in V$, $c\in M$, $k\in\bZ$, $m,n\in\frac1N\bZ$:
\begin{equation}\label{vert13}
\begin{split}
\sum_{j=0}^\infty & \binom{m}{j} (a_{(k+j)}b)_{(m+n-j)}^M c
= \sum_{i=0}^\infty \binom{k}{i} (-1)^i
a_{(m+k-i)}^M(b_{(n+i)}^Mc)
\\
&- \sum_{i=0}^\infty \binom{k}{i} (-1)^{k+i} \, b_{(n+k-i)}^M(a_{(m+i)}^Mc),
\end{split}
\end{equation}
provided that $\sigma a = e^{-2\pi\ii m} a$.

In particular, we have the \emph{commutator formula} for $a,b\in V$ and $m,n\in\frac1N\bZ$ such that $\sigma a = e^{-2\pi\ii m} a$:
\begin{equation}\label{vert14}
[a_{(m)}^M, b_{(n)}^M]=\sum_{j=0}^\infty \binom{m}{j} (a_{(j)}b)^M_{(m+n-j)}.
\end{equation}
The translation covariance properties \eqref{vert2} and \eqref{vert2b} remain valid for twisted modules.
However, formula \eqref{vert4} does not hold for twisted modules; it is replaced by \cite[Eq. (14.16)]{bombay}.

Let $V_1$ and $V_2$ be vertex algebras with finite-order automorphisms $\sigma_1$ and $\sigma_2$, respectively.
Then $\sigma_1\otimes\sigma_2$ is an automorphism of the tensor product vertex algebra $V_1\otimes V_2$.
Given $\sigma_i$-twisted $V_i$-modules $M_i$ ($i=1,2$), their tensor product $M_1\otimes M_2$
is a $\sigma_1\otimes\sigma_2$-twisted module of $V_1\otimes V_2$ with
\begin{equation}\label{vert15}
Y^{M_1\otimes M_2}(a\otimes b, z)=Y^{M_1}(a, z)\otimes Y^{M_2}(b,z),
\qquad a\in V_1, \; b\in V_2,
\end{equation}
see \cite{BE,FHL}.

\subsection{Lattice vertex algebras}\label{sec:lava}
Let $Q$ be an integral \emph{lattice} of rank $\ell$, i.e., a free abelian group of rank $\ell$ with a symmetric bilinear form $(\cdot|\cdot)\colon Q\times Q \rightarrow \mathbb{Z}.$ 
We will assume that $Q$ is even, which means that $|\alpha|^2=(\alpha|\alpha)\in 2\bZ$ for all $\alpha\in Q$.
Let $\fh=\mathbb{C}\otimes_{\mathbb{Z}} Q$ and extend the form $(\cdot|\cdot)$ to $\fh$ using bilinearity.

The \emph{Heisenberg Lie algebra} $\hat{\fh}$ is defined as $$\hat{\fh}=(\fh\otimes\mathbb{C}[t,t^{-1}])\oplus \mathbb{C}K$$
with the Lie brackets ($h, h' \in \fh$, $m,n\in \mathbb{Z}$):
\begin{equation}\label{heisbr}
[h\otimes t^m, h' \otimes t^n]=m\delta_{m,-n}(h|h')K, \qquad
[\hat{\fh}, K]=0.
\end{equation}
We will use the notation $h_{(m)} = h\otimes t^m$.
The Lie algebra $\hat{\fh}$ has a unique highest-weight representation on the so-called bosonic \emph{Fock space}
$$B=\Ind^{\hat{\fh}}_{\fh[t]\oplus\mathbb{C}K}\mathbb{C} \simeq S(t^{-1}\fh[t^{-1}]),$$
where $K$ acts as $I$ and $\fh[t]$ acts trivially on $\mathbb{C}$. 
The Fock space $B$ has the structure of a vertex algebra with a vacuum vector $\vac$ the highest-weight vector
and the state-field correspondence $Y$ defined as follows.
For $h\in\fh$, we identify $h$ with $h_{(-1)}\vac\in B$ and let
\begin{equation}\label{heis1}
Y(h,z)=\sum_{m\in \bZ}h_{(m)}z^{-m-1},
\qquad h\in\fh
\end{equation}
be the free boson fields.
All other fields in $B$ are obtained from them by applying repeatedly formula \eqref{vert4}; see \cite{curves,VAFB,li}.

There exists a bimultiplicative function $\varepsilon\colon Q\times Q\rightarrow \{\pm 1\}$ such that
\begin{equation}\label{eps}
\varepsilon(\alpha, \alpha)=(-1)^{|\alpha|^2(|\alpha|^2+1)/2},
\qquad \alpha \in Q.
\end{equation}
By bimultiplicativity, $\varepsilon$ satisfies
$$
\varepsilon(\alpha, \beta)\varepsilon(\beta, \alpha)=(-1)^{(\alpha|\beta)},
\qquad \alpha,\beta \in Q.
$$
We can use $\varepsilon$ to define the twisted group algebra $\bC_\varepsilon[Q]$ with basis $\{e^\alpha \}_{\alpha \in Q}$ and multiplication
$$e^\alpha e^\beta = \varepsilon(\alpha, \beta)e^{\alpha + \beta}.$$
The representation of $\hat{\fh}$ can be extended to the space $$V_Q=B\otimes \bC_\varepsilon[Q]$$ 
by the action
$$h_{(m)}( u\otimes e^\beta ) = (h_{(m)}u+\delta_{m,0}(h|\beta )u)\otimes e^\beta $$
for $h\in \fh$, $m\in \bZ$, $u \in B$ and $\beta \in Q$. 
In particular, note that $e^\beta$ is a highest-weight vector for the Heisenberg Lie algebra:
\begin{equation}\label{heishw}
h_{(m)} e^\beta = \delta_{m,0}(h|\beta)e^\beta, \qquad m\ge0, \;\; h\in \fh, \;\; \beta \in Q.
\end{equation}
We can also represent the algebra $\bC_\varepsilon[Q]$ on $V_Q$ by
$$e^\alpha(u\otimes e^\beta)=\varepsilon(\alpha, \beta)(u\otimes e^{\alpha+\beta})$$
for $ u \in B$ and $\alpha,\beta \in Q$.

For simplicity of notation, we will write $e^\alpha$ for $\vac\otimes e^\alpha \in V_Q$ and $h$ for $h_{(-1)}\vac \otimes e^0 \in V_Q$,
where $\alpha\in Q$ and $h\in\fh$. 
The space $V_Q$ has the structure of a vertex algebra called the \emph{lattice vertex algebra}, with a vacuum vector $\vac\otimes e^0$ 
and a state-field correspondence generated by the free boson fields \eqref{heis1} and the so-called vertex operators
\begin{equation}\label{heis2}
Y(e^\alpha,z)=e^\alpha z^{\alpha_{(0)}} \exp\Biggl(\sum_{n=1}^\infty \alpha_{(-n)}\frac{z^n}{n}\Biggr)\exp\Biggl(\sum_{n=1}^\infty \alpha_{(n)}\frac{z^{-n}}{-n}\Biggr).
\end{equation}
In this formula, $z^{\alpha_{(0)}}$ acts on $V_Q$ by 
$$z^{\alpha_{(0)}} (u\otimes e^\beta)=z^{(\alpha|\beta)}(u\otimes e^\beta), 
\qquad u\in B, \; \alpha,\beta \in Q.$$
For future use, we also recall the action of the translation operator:
\begin{equation}\label{Teal}
T e^\alpha = \alpha_{(-1)} e^\alpha, \qquad \alpha\in Q.
\end{equation}

Suppose $\sigma\in \Aut(Q)$ where $\sigma^N=I$ and extend $\sigma$ to $\fh$ by linearity. 
Then $\sigma$ lifts to an automorphism of the Heisenberg Lie algebra $\hat{\fh}$ by
$\sigma(h_{(m)}) = (\sigma h)_{(m)}$, and to an automorphism of the vertex algebra $B$ so that $\sigma\vac=\vac$.
Since the cocycles $\varepsilon(\alpha, \beta)$ and $\varepsilon(\sigma \alpha, \sigma \beta)$ are equivalent, there is a map $\eta\colon Q\rightarrow \{\pm 1\}$ such that 
$$\eta(\alpha)\eta(\beta)\varepsilon(\alpha, \beta)=\eta(\alpha+\beta)\varepsilon(\sigma \alpha, \sigma \beta),
\qquad \alpha,\beta\in Q.$$ 
The map $\eta$ can be chosen so that 
$\eta(\alpha)=1$ if $\sigma\alpha=\alpha$ (see \cite{BE}).
We can lift $\sigma$ to an automorphism of $V_Q$ by $\sigma(e^\alpha)=\eta(\alpha) e^{\sigma\alpha}$.
Notice that the order of the lift $\sigma\in\Aut(V_Q)$ is $N$ or $2N$.
The irreducible $\sigma$-twisted $V_Q$-modules were classified in \cite{bak} (see also \cite{dong,lep1}).

\section{Vertex operator representations of affine Kac--Moody algebras}\label{sec:voraff}
In this section, we construct representations of affine Kac--Moody algebras from twisted modules over affine vertex algebras. 

\subsection{Affine Kac--Moody algebras}\label{sec:km}

Let ${\fg}$ be a simple finite-dimensional Lie algebra of type $X_{\ell}$ where $X=A, B, C, \dotsc, G$. 
Recall that the \emph{affine Kac-Moody algebra} of type $X_{\ell}^{(1)}$ is defined as \cite{IDLA}:
$$\hat{\L}(\fg)= (\fg \otimes \bC[t, t^{-1}])\oplus \bC K \oplus \bC d,$$ with the Lie brackets
\begin{equation}\label{L(g)}
\begin{split}
[a\otimes t^m, b\otimes t^n]&=[a,b]\otimes t^{m+n}+m\delta_{m,-n}(a|b)K,\\
[K,\hat{\L}(\fg)]&=0,\\
[d, a \otimes t^m]&=m a\otimes t^m,
\end{split}
\end{equation}
for $a,b\in\fg$ and $m,n\in\bZ$.
Here $(\cdot|\cdot)$ is a nondegenerate symmetric invariant bilinear form on $\fg$, such as a scalar multiple of
the Killing form.
We will denote by $\hat\fg$ or $\hat{\L}'(\fg)$ the subalgebra of $\hat{\L}(\fg)$ given by $$\hat\fg=\hat{\L}'(\fg)= (\fg \otimes \bC[t, t^{-1}])\oplus \bC K.$$ 

Let  $\sigma$ be an automorphism of the Lie algebra $\fg$ of finite order $N$. We can define a subalgebra $\hat{\L}(\fg,\sigma)$ of $\hat{\L}(\fg)$ by
\begin{equation}\label{Lgs}
\hat{\L}(\fg,\sigma)=\bigoplus_{j \in \bZ}\L(\fg, \sigma)_{j}\oplus \bC K \oplus \bC d,\\
\end{equation}
where
\begin{equation}\label{gj}
\begin{split}
\L(\fg, \sigma)_{j}&=\fg_{j}  \otimes \bC t^j, \qquad j\in\bZ, \\
\fg_{j} &= \{ a\in\fg \,|\, \sigma a = e^{2\pi\ii j/N} a \}.
\end{split}
\end{equation}
When $\fg$ is simply laced (of type $X=A,D,E$) and $\sigma$ is a diagram automorphism of order $N=2$ or $3$, then the Lie algebra $\hat{\L}(\fg,\sigma)$ is known as the \textit{twisted affine Kac--Moody algebra} of type $X_{\ell}^{(N)}.$
We will denote by $\hat{\L}'(\fg,\sigma)$ the subalgebra of $\hat{\L}(\fg,\sigma)$ given by $$\hat{\L}'(\fg, \sigma)=\bigoplus_{j \in \bZ}\L(\fg, \sigma)_{j}\oplus \bC K.$$ 

If $\sigma$ is an arbitrary automorphism of $\fg$ (simply or non-simply laced) of finite order $N$, then there exists an associated diagram automorphism $\mu$ (where the order of $\mu$ is $1,2,$ or $3$, depending on $\fg$) and an inner automorphism $\varphi$ such that $\sigma=\mu\varphi$. Then the Lie algebra $\hat{\L}(\fg,\sigma)$ is isomorphic to $\hat{\L}(\fg,\mu)$; see \cite[Proposition 8.1, Theorem 8.5]{IDLA}.

\subsection{Affine vertex algebras}\label{sec:affva}
Suppose $\fg$ is a simple finite-dimensional Lie algebra, equipped with a nondegenerate symmetric invariant bilinear form $(\cdot|\cdot)$.
Consider the Lie algebra  $\hat\fg=\hat{\L}'(\fg)$  with the brackets given by the first two equations in \eqref{L(g)}. 
For a fixed $k\in \bC$ (called the \emph{level}),
consider the (generalized) Verma module for $\hat\fg$:
$$V_k(\fg)=\Ind^{\hat{\fg}}_{\fg[t]\oplus\mathbb{C}K}\mathbb{C},
\qquad $$
where $\fg[t]$ acts as $0$ on $\bC$ and $K$ acts as multiplication by $k$.
The module $V_k(\fg)$ has the structure of a vertex algebra \cite{fz}, called the 
\emph{universal affine vertex algebra} at level $k$.
The $\hat\fg$-module $V_k(\fg)$ has a unique irreducible quotient $V^k(\fg)$,
which is also a vertex algebra \cite{fz}, known as the 
\emph{simple affine vertex algebra} at level $k$.

Let us review the vertex algebra structure of $V=V_k(\fg)$; the same applies to $V=V^k(\fg)$ as well.
The vacuum vector $\vac$ is the highest-weight vector of the $\hat\fg$-module $V$.
We shall not review the full state-field correspondence $Y$ but only the generating fields.
For $a\in\fg$ and $n\in\bZ$, let $a_{(n)}$ act as $a\otimes t^n$ on $V$.
We embed $\fg$ in $V$ so that we identify $a\in\fg$ with  $a_{(-1)} \vac\in V$. 
Then we have the fields
$$Y(a,z)=\sum_{n\in \bZ}a_{(n)}z^{-n-1}, \qquad a\in \fg,$$
known as \emph{currents}.
All other fields in $V$ are obtained from them by applying repeatedly formula \eqref{vert4}; see \cite{curves,VAFB,li}.

For $a,b\in\fg\subset V$, their modes satisfy the commutation relations of the Lie algebra $\hat\fg$:
\begin{equation}\label{aff3}
[a_{(m)}, b_{(n)}]=[a,b]_{(m+n)}+m\delta_{m,-n} (a|b)k.
\end{equation}
By the commutator formula \eqref{vert6}, this is equivalent to the $j$-th products 
\begin{equation}\label{aff4}
a_{(0)}b=[a,b], \qquad
a_{(1)}b=(a|b)k\vac, \qquad
a_{(j)}b=0 \;\; (j\geq 2).
\end{equation}
One can show that a $V_k(\fg)$-module is the same as a $\hat\fg$-module $M$ with the property that
$a_{(n)} v =0$ for $a\in\fg$, $v\in M$ and $n\gg0$ (see \cite{fz,bombay,li}).
In the next subsection,
we will obtain a similar result for twisted $V_k(\fg)$-modules
(cf.\ \cite{bombay,KT}).

 \subsection{Twisted modules over $V_k(\fg)$ and twisted affine Lie algebras}\label{sec:twaf}
As in Section \ref{sec:km},
let ${\fg}$ be a simple finite-dimensional Lie algebra and $\sigma\in\Aut(\fg)$ such that $\sigma^N=I$. 
We can extend $\sigma$ uniquely to an automorphism of the universal affine vertex algebra $V_k(\fg)$ by 
\begin{equation}\label{aff5}
\sigma(\vac)=\vac, \quad
{\sigma}(a_{(m)})=(\sigma a)_{(m)}, \qquad a\in\fg, \; m\in \bZ.
\end{equation}

\begin{proposition}\label{prop:aff}
For any\/ $\sigma$-twisted\/ $V_k({\fg})$-module\/ $M$,
the Lie algebra spanned by the modes of the fields\/ $Y^M(a,z)$ for\/ $a \in {\fg}$ form a representation of the twisted affine Kac--Moody algebra\/ $\hat{\L}'({\fg},\sigma^{-1})$ on\/ $M$
of level\/ $k$.
\end{proposition}

\begin{proof}
We will define a Lie algebra homomorphism from $\hat{\L}'({\fg},\sigma^{-1})$ to the modes of the fields $Y^M(a,z)$. By linearity, since $\sigma$ is diagonalizable,
we can assume that $a \in {\fg}$ is an eigenvector of $\sigma$. Suppose that $\sigma a=e^{-2\pi\ii j/N}a$ for some $j\in\bZ$.
Then by \eqref{gj}, we have $a\otimes t^j \in \hat{\L}({\fg},\sigma^{-1})_j$.
On the other hand, the property \eqref{vert12} of twisted modules implies that in the expansion \eqref{vert11} of $Y^M(a,z)$, we only have modes $a^M_{(m)}$
such that $m\in\frac{j}N+\bZ$. We define a linear map from $\hat{\L}'({\fg},\sigma^{-1})$ to the span of these modes, by sending $a\otimes t^j$ to $a^M_{(j/N)}$.
We also send $K$ to $kI$ as a linear operator on $M$.

To check that this is a Lie algebra homomorphism, we compute the commutator of modes, which is given by the $\sigma$-twisted commutator formula \eqref{vert14}.
Using \eqref{aff4}, we obtain
$$[a^M_{(m)}, b^M_{(n)}]=[a,b]^M_{(m+n)}+m\delta_{m,-n} (a|b)k,$$
for $a,b\in\fg$ and $m,n\in \frac1N\bZ$ such that $\sigma a=e^{-2\pi\ii m}a$ and $\sigma b=e^{-2\pi\ii n}b$.
As this coincides with the Lie bracket \eqref{L(g)} in $\hat{\L}'({\fg},\sigma^{-1})$ with $K=kI$, the claim follows.
\end{proof}
\begin{remark}
When $\sigma$ is an inner automorphism of $\fg$, then $\hat{\L}'({\fg},\sigma^{-1}) \simeq \hat{\L}'({\fg}) = \hat\fg$
is an untwisted affine Lie algebra. 
\end{remark}

Suppose now $\fg$ is a simple Lie algebra of type $X_\ell$ where $X=A, D, E$ (simply laced). Let $\Delta$ be its root system and $Q=\bZ\Delta$ its root lattice. 
Let $\sigma$ be an automorphism of the lattice $Q$ of finite order $N$. 
Recall from Section \ref{sec:lava} that we can use $Q$ to construct the lattice vertex algebra $V_Q$, and we can lift $\sigma$ to an automorphism of the lattice vertex algebra $V_Q$. 

\begin{corollary}[cf.\ \cite{112, bombay}] \label{cor112}
For any $\sigma$-twisted\/ $V_Q$-module\/ $M$,
the modes of the free bosons\/ $Y^M(h,z)$ for\/ $h \in \fh$ and the modes of the vertex operators\/ $Y^M(e^\alpha,z)$ for\/ $\alpha \in \Delta$ 
span a representation of the twisted affine Kac--Moody algebra\/ $\hat{\L}'({\fg},\sigma^{-1})$ on\/ $M$. 
\end{corollary}
\begin{proof}
By the Frenkel--Kac construction \cite{fk}, the lattice vertex algebra $V_Q$ is isomorphic to the simple affine vertex algebra $V^1({\fg})$ (see \cite[Theorem 5.6 (c)]{VAFB}). 
The map $\sigma$ induces an automorphism of $\fg$ and hence of $V^1({\fg})$. Recall that $V^1({\fg})$ is a quotient of the universal affine vertex algebra $V_1({\fg})$.
Thus any $\sigma$-twisted $V^1({\fg})$-module $M$ is also a $\sigma$-twisted module for $V_1({\fg})$. The claim then follows from Proposition \ref{prop:aff}.
\end{proof}

\begin{remark}
If $\fg_0$ is a non-simply laced Lie algebra, then $\fg_0$ can be embedded in a simply laced Lie algebra $\fg$ as the set of fixed points under a diagram automorphism $\mu$ of $\fg$.
We can use Corollary \ref{cor112} to construct representations of $\hat{\L}'({\fg},\sigma^{-1})$ and then restrict them to the subalgebra $\hat{\L}'(\fg_0)$; see \cite{BT, GNOS}. 
Level-one representations of affine Kac--Moody algebras associated to non-simply laced Lie algebras have been explored in papers such as \cite{kroode, misra1, misra2, xj}. 
\end{remark}

\section{Vertex operator representations of toroidal Lie algebras}\label{sec:vortor}
In this section, we construct representations of toroidal Lie algebras from twisted modules over a tensor product of an affine vertex algebra and a certain lattice vertex algebra.

\subsection{Untwisted toroidal Lie algebras}\label{sec:2.1}
Through the rest of the paper, we fix a positive integer $r$ and a simple finite-dimensional Lie algebra $\fg$ equipped with a nondegenerate symmetric invariant bilinear form $(\cdot|\cdot)$.
We will use variables $t_0,t_1,\dots,t_r$ and multi-index notation
$$
\bt=(t_1,\dots,t_r), \quad \bm=(m_1,\dots,m_r) \in \bZ^r, \quad \bt^\bm = t_1^{m_1}\cdots t_{r}^{m_{r}}.
$$
Consider the loop algebra $\L_{r+1}(\fg)$ in $r+1$ variables
$$\L_{r+1}(\fg)=\fg\otimes\O, \qquad \O=\mathbb{C}[t_0^{\pm1},t_1^{\pm 1},\dots ,t_{r}^{\pm 1}],$$
 with the Lie bracket given by ($a, b \in \fg$, $m_0, n_0\in \bZ$, $\bm,\bn\in \bZ^r$):
\begin{align*}[a\otimes t_0^{m_0}\bt^\bm, b \otimes t_0^{n_0}\bt^\bn]&=[a,b]\otimes t_0^{m_0+n_0}\bt^{\bm+\bn} .
\end{align*}
Next, we create a central extension $\hat{\L}'_{r+1}(\fg)$ of $\L_{r+1}(\fg)$: 
$$\hat{\L}'_{r+1}(\fg)=\L_{r+1}(\fg)\oplus \K,$$ 
where
\begin{equation}\label{defk}
\K=\Bigl(\bigoplus_{i=0}^{r} \bC K_i\otimes \O \Bigr) \Big/ 
\Span_{\bC}\Bigl\{\sum_{i=0}^{r}m_i K_i\otimes t_0^{m_0} \bt^\bm \,\Big|\, m_i\in \bZ\Bigr\}.
\end{equation}
The Lie brackets in $\hat{\L}'_{r+1}(\fg)$ are given by:
\begin{align}\label{2.1}
&[a \otimes t_0^{m_0}\bt^\bm,b \otimes t_0^{n_0}\bt^\bn]=[a,b]\otimes t_0^{m_0+n_0}\bt^{\bm+\bn}\\
\nonumber
&\quad \quad \quad \quad \quad\quad\quad\quad\quad\quad+(a|b)\sum_{i=0}^{r}m_i K_i\otimes t_0^{m_0+n_0}\bt^{\bm+\bn},
\\ \label{2.2}
&[\K, \hat{\L}'_{r+1}(\fg)]=0.
\end{align}
By Kassel's Theorem, the central extension $\hat{\L}'_{r+1}(\fg)$ of $\L_{r+1}(\fg)$ is universal \cite{kassel} 
(setting $K_i=t_i^{-1} dt_i$ allows us to identify $\K$ with $\O/d\O$).
The Lie algebra $\hat{\L}'_{r+1}(\fg)$ is known as the (untwisted) \emph{toroidal Lie algebra} \cite{mry}. 

It will be convenient to slightly modify the definition of $\hat{\L}'_{r+1}(\fg)$ as follows.
For a given complex number $k$ (called the \emph{level}), we replace the bracket \eqref{2.1} with:
\begin{equation}\label{2.3}
\begin{split}
&[a \otimes t_0^{m_0}\bt^\bm,b \otimes t_0^{n_0}\bt^\bn]=[a,b]\otimes t_0^{m_0+n_0}\bt^{\bm+\bn}\\
&\quad \quad \quad \quad \quad\quad\quad\quad\quad\quad+k(a|b) \sum_{i=0}^{r}m_i K_i\otimes t_0^{m_0+n_0}\bt^{\bm+\bn}.
\end{split}
\end{equation}
The resulting Lie algebra $\L_{r+1}(\fg)\oplus \K$ with brackets \eqref{2.2}, \eqref{2.3} will be denoted as $\hat{\L}'_{r+1,k}(\fg)$
and called the (untwisted) \emph{toroidal Lie algebra of level $k$}. Notice that, for $k\ne0$, formulas \eqref{2.1} and \eqref{2.3} 
are equivalent after rescaling the bilinear form $(\cdot|\cdot)$ or rescaling the central elements $K_0,\dots,K_r$.

Now we will add derivations to our toroidal Lie algebra, as in \cite{yb}. 
We let $\D$ be the Lie algebra of derivations of $\O$ given by 
$$\D= \Bigl\{\sum_{i=0}^{r}d_{i}\otimes f_i \,\Big|\, f_i \in \O \Bigr\}, \qquad d_i=t_i \frac{\partial}{\partial t_i} \,,$$
and $\D_+$ be the subalgebra of $\D$ given by
$$\D_+= \Bigl\{\sum_{i=1}^{r}d_{i}\otimes f_i \,\Big|\, f_i \in \O \Bigr\}.$$
Then the elements of $\D_+$ extend uniquely to derivations of the Lie algebra $\hat{\L}'_{r+1}(\fg)$ by ($a \in \fg$, $1\le i,j\le r$):
\begin{align*}
(d_i\otimes t_0^{m_0}\bt^\bm) (a\otimes t_0^{n_0}\bt^\bn) &=n_i a\otimes t_0^{m_0+n_0}\bt^{\bm+\bn}, \\
(d_i\otimes t_0^{m_0}\bt^\bm) (K_j\otimes t_0^{n_0}\bt^\bn) &=n_i K_j\otimes t_0^{m_0+n_0}\bt^{\bm+\bn} \\
\nonumber
&\quad \quad \quad  +\delta_{i,j}\sum_{l=0}^{r}m_l K_l \otimes t_0^{m_0+n_0}\bt^{\bm+\bn}.
\end{align*}


The Lie algebra we will consider in this paper, which we will refer to again as a \emph{toroidal Lie algebra of level $k$}, is  $$\hat{\L}_{r+1,k}(\fg)=\hat{\L}'_{r+1,k}(\fg)\oplus \D_+$$
with the Lie brackets given by \eqref{2.2}, \eqref{2.3} and ($a \in \fg$, $m_l, n_l \in \bZ$, $1\le i,j\le r$):
\begin{align} \label{d1}
[d_{i}\otimes t_0^{m_0}\bt^\bm, a\otimes t_0^{n_0}\bt^\bn]&=n_i a \otimes  t_0^{m_0+n_0}\bt^{\bm+\bn},\\
\label{d2}
[d_{i}\otimes t_0^{m_0}\bt^\bm, K_j\otimes t_0^{n_0}\bt^\bn]&=n_i K_j\otimes t_0^{m_0+n_0}\bt^{\bm+\bn}   \\
\nonumber
&\quad +\delta_{i,j}\sum_{l=0}^{r}m_l K_l\otimes t_0^{m_0+n_0}\bt^{\bm+\bn},\\
\label{d3}
[d_{i}\otimes t_0^{m_0}\bt^\bm, d_{j}\otimes t_0^{n_0}\bt^\bn]&=n_i d_j \otimes t_0^{m_0+n_0}\bt^{\bm+\bn} \\
\nonumber
 - \, m_j d_i\otimes t_0^{m_0+n_0}\bt^{\bm+\bn} &-n_i m_j \sum_{l=0}^{r} m_l K_l\otimes t_0^{m_0+n_0}\bt^{\bm+\bn}.
\end{align}
The last term in \eqref{d3} corresponds to a $\K$-valued $2$-cocycle on $\D_+$. For more information on derivations of toroidal Lie algebras and $2$-cocycles, see \cite{sbm}. 
Notice that, for $k\ne0$, if we rescale the generators $K_0,\dots,K_r$ in order to replace \eqref{2.3} with \eqref{2.1}, the $2$-cocycle gets rescaled by $1/k$.

\subsection{Twisted toroidal Lie algebras}\label{sec:2.2}
As before, fix a level $k\in\bC$, and let $\sigma$ be an automorphism of order $N$ of a simple finite-dimensional Lie algebra $\fg$.
As in \eqref{gj}, we denote by $\fg_j$ $(j\in\bZ)$ the eigenspace of $\sigma$ with eigenvalue $e^{2\pi\ii j/N}$.
Note that the $\sigma$-invariance of $(\cdot|\cdot)$ implies
\begin{equation} \label{eigenorth}
(a|b)=0, \qquad a\in\fg_m, \;\; b\in\fg_n, \;\; m+n\not\equiv 0 \mod N.
\end{equation}
Consider the subalgebra $\L_{r+1}(\fg,\sigma)$ of the loop algebra $\L_{r+1}(\fg)$ given by
$$\L_{r+1}(\fg,\sigma)=\bigoplus_{m_0 \in \bZ}\L_{r+1}(\fg, \sigma)_{m_0},$$
where
$$\L_{r+1}(\fg, \sigma)_{m_0}=\fg_{{m}_0}  \otimes \Span_\bC \bigl\{t_0^{m_0} \bt^\bm\, \big| \, \bm\in \bZ^r \bigr\}.
$$
Let 
$$\K'= \frac{ \Span_{\bC}\bigl\{K_i\otimes t_0^{Nm_0} \bt^\bm \, \big| \, m_0\in \bZ, \; \bm\in \bZ^r,\; i=0,\dots,r \bigr\} }
{ \Span_{\bC}\bigl\{ \bigl(N m_0 K_0 + \sum_{i=1}^{r}m_i K_i \bigr)\otimes t_0^{N m_0} \bt^\bm \,\big|\, m_i\in \bZ\bigr\} }\,.$$
We can identify $\K'$ as the subspace of $\K$ given by the image of
$$\Span_{\bC}\bigl\{K_i\otimes t_0^{Nm_0} \bt^\bm \, \big| \, m_0\in \bZ, \; \bm\in \bZ^r,\; i=0,\dots,r \bigr\}$$
under the quotient map $\bigoplus_{i=0}^{r} \bC K_i\otimes \O \to \K$ (cf.\ \eqref{defk}).
Then the central extension 
$$\hat{\L}_{r+1,k}'(\fg, \sigma)=\L_{r+1}(\fg, \sigma)\oplus \K' \subset \hat{\L}_{r+1,k}'(\fg)$$ is a subalgebra of $\hat{\L}_{r+1,k}'(\fg)$,
thanks to \eqref{2.3}, \eqref{eigenorth}.
When $\sigma$ is a diagram automorphism, $\hat{\L}_{r+1,1}'(\fg, \sigma)$ is known as the \emph{twisted toroidal Lie algebra} \cite{FJ}.
We will continue to use that terminology for an arbitrary finite-order automorphism $\sigma$ of $\fg$.

As in Section \ref{sec:2.1}, we can add to $\hat{\L}_{r+1,k}'(\fg, \sigma)$ a subalgebra of the Lie algebra of derivations $\D$ of $\O$. We define
\begin{equation}\label{ttla}
\hat{\L}_{r+1,k}(\fg, \sigma)=\hat{\L}_{r+1,k}'(\fg, \sigma) \oplus \D'_+
\end{equation}
where
$$\D'_+=\Span_{\bC}\bigl\{d_{i}\otimes t_0^{Nm_0}\bt^\bm \, \big| \, m_0\in \bZ, \; \bm\in \bZ^r,\;  i=1,\dots,r \bigr\}.$$
It is easy to see from \eqref{d1}--\eqref{d3} that $\hat{\L}_{r+1,k}(\fg, \sigma)$ is a subalgebra of the toroidal Lie algebra $\hat{\L}_{r+1,k}(\fg)$.
We will also call $\hat{\L}_{r+1,k}(\fg, \sigma)$ the \emph{twisted toroidal Lie algebra of level $k$}.

\subsection{Twisted modules over $V_k(\fg)\otimes V_J$ and twisted toroidal Lie algebras}
We will now explore the relationship between twisted toroidal Lie algebras in $r+1$ variables and the tensor product of the universal affine vertex algebra $V_k(\fg)$
with the lattice vertex algebra corresponding to $r$ copies of a certain rank-$2$ lattice. The level $k\in\bC$ will be fixed through the end of the section.

As before, consider an automorphism $\sigma$ of finite order $N$ of a simple finite-dimensional Lie algebra $\fg$, and the twisted toroidal Lie algebra (of level $k$) $\hat{\L}_{r+1,k}(\fg, \sigma)$ defined by \eqref{ttla}.
For $i=1,\dots,r$, let $J_i$ be the lattice given by 
\begin{equation}\label{J_i}
J_i=\bZ\delta^i\oplus \bZ \Lambda^i_0, \qquad
(\delta^i | \Lambda^i_0)=1, \quad (\delta^i|\delta^i)=(\Lambda^i_0|\Lambda^i_0)=0.
\end{equation} 
We define a bimultiplicative function $\varepsilon\colon J_i\times J_i\rightarrow \{\pm 1\}$ satisfying \eqref{eps} by $\varepsilon(\delta^i, \Lambda^i_0)=-1$ and $\varepsilon=1$ for all other pairs of generators. 
Then we can form the lattice vertex algebra $V_{J_i}$ as in Section \ref{sec:lava}.

Introduce the orthogonal direct sum \begin{equation}\label{J2}J=J_1\oplus\dots\oplus J_r,\end{equation} and extend $\varepsilon$ to $J\times J$ by $\varepsilon (\delta^i, \Lambda^j_0)=-1$ for $i = j$  and $\varepsilon=1$ for all other pairs of generators. 
Then the lattice vertex algebra $V_J$ is isomorphic to the tensor product:
\begin{equation}\label{J3}V_J\simeq V_{J_1} \otimes\dots\otimes V_{J_r}.\end{equation}
As preparation for our main theorem, we need to calculate some $n$-th products in $V_J$. 
We will use the notation
$$\bp\bd=\sum_{i=1}^{r}p_i \delta^i, \qquad \bp=(p_1,\dots,p_r) \in \bZ^r.$$

\begin{lemma}\label{lem:vjprod}
The lattice vertex algebra\/ $V_J$ has the following $n$-th products for\/ $\bp,\bq\in\bZ^r$ and\/ $i,j= 1, \dots, r\colon$
\begin{align*} 
e^{\bp\bd}_{(-1)}e^{\bq\bd}&=e^{(\bp+\bq)\bd}, \\
e^{\bp\bd}_{(-2)}e^{\bq\bd}&=(\bp\bd)_{(-1)}e^{(\bp+\bq)\bd},\\
\bigl(u_{(-1)}e^{\bp\bd}\bigr)_{(n)}\bigl(v_{(-1)}e^{\bq\bd}\bigr)&=0 \;\;\text{ for }\; n \geq 0, \;\; u,v\in\{\vac,\delta^1,\dots,\delta^r\}, \\
\bigl({\Lambda^i_0}_{(-1)}e^{\bp\bd}\bigr)_{(0)}e^{\bq\bd}&=q_i\,e^{(\bp+\bq)\bd}, \\ \bigl({\Lambda^i_0}_{(-1)}e^{\bp\bd}\bigr)_{(n)}e^{\bq\bd}&=0 \;\;\text{ for }\; n \geq 1,\\
\bigl({\Lambda^i_0}_{(-1)}e^{\bp\bd}\bigr)_{(0)}\bigl(\delta^j_{(-1)}e^{\bq\bd}\bigr)&=q_i\delta^j_{(-1)}e^{(\bp+\bq)\bd}+\delta_{i,j}(\bp\bd)_{(-1)}e^{(\bp+\bq)\bd},\\
\bigl({\Lambda^i_0}_{(-1)}e^{\bp\bd}\bigr)_{(1)}\bigl(\delta^j_{(-1)}e^{\bq\bd}\bigr)&=\delta_{i,j}e^{(\bp+\bq)\bd},\\
\bigl({\Lambda^i_0}_{(-1)}e^{\bp\bd}\bigr)_{(n)}\bigl(\delta^j_{(-1)}e^{\bq\bd}\bigr)&=0 \;\;\text{ for }\; n \geq 2,\\
\bigl({\Lambda^i_0}_{(-1)}e^{\bp\bd}\bigr)_{(0)}\bigl({\Lambda^j_0}_{(-1)}e^{\bq\bd}\bigr)&=\bigl(-p_j{\Lambda^i_0}_{(-1)}+q_i{\Lambda^j_0}_{(-1)}-q_ip_j(\bp\bd)_{(-1)}\bigr)e^{(\bp+\bq)\bd},\\
\bigl({\Lambda^i_0}_{(-1)}e^{\bp\bd}\bigr)_{(1)}\bigl({\Lambda^j_0}_{(-1)}e^{\bq\bd}\bigr)&=-q_ip_je^{(\bp+\bq)\bd},\\
\bigl({\Lambda^i_0}_{(-1)}e^{\bp\bd}\bigr)_{(n)}\bigl({\Lambda^j_0}_{(-1)}e^{\bq\bd}\bigr)&=0 \;\;\text{ for }\; n \geq 2.
\end{align*} 
\end{lemma}
\begin{proof}
The $n$-th product $\bigl({\Lambda^i_0}_{(-1)}e^{\bp\bd}\bigr)_{(n)}\bigl({\Lambda^j_0}_{(-1)}e^{\bq\bd}\bigr)$ is the coefficient of $z^{-n-1}$ in $Y\bigl({\Lambda^i_0}_{(-1)}e^{\bp\bd},z \bigr){\Lambda^j_0}_{(-1)}e^{\bq\bd}.$ 
Applying the $(-1)$-st product identity \eqref{vert3} and the definition \eqref{heis2} of vertex operators, we find:
\begin{align*}
Y(&{\Lambda^i_0}_{(-1)}e^{\bp\bd},z){\Lambda^j_0}_{(-1)}e^{\bq\bd}\\
&={:}Y(\Lambda^i_0,z)Y(e^{\bp\bd},z){:} \,{\Lambda^j_0}_{(-1)}e^{\bq\bd}\\
&=\Biggl(\sum_{l<0}{\Lambda^i_0}_{(l)}z^{-l-1}\Biggr)e^{\bp\bd} z^{(\bp\bd)_{(0)}}\exp\Biggl(\sum_{m>0} (\bp\bd)_{(-m)}\frac{z^m}{m}\Biggr) \times \\
&\quad \quad \quad \quad \quad \quad \quad \quad \times\exp\Biggl(\sum_{m>0} (\bp\bd)_{(m)}\frac{z^{-m}}{-m}\Biggr){\Lambda^j_0}_{(-1)}e^{\bq\bd}\\
&+e^{\bp\bd} z^{(\bp\bd)_{(0)}}\exp\Biggl(\sum_{m>0} (\bp\bd)_{(-m)}\frac{z^m}{m}\Biggr) \exp\Biggl(\sum_{m>0} (\bp\bd)_{(m)}\frac{z^{-m}}{-m}\Biggr) \times\\
&\quad \quad \quad \quad \quad \quad \quad \quad \times \Biggl(\sum_{l\geq 0}{\Lambda^i_0}_{(l)}z^{-l-1}\Biggr){\Lambda^j_0}_{(-1)} e^{\bq\bd} .
\end{align*}
Next, we use the brackets \eqref{heisbr} in the Heisenberg Lie algebra and its action \eqref{heishw} on its highest-weight vectors, to get:
$$
[(\bp\bd)_{(m)}, {\Lambda^j_0}_{(-1)}] = \delta_{m,1} p_j, \qquad
[{\Lambda^i_0}_{(l)}, {\Lambda^j_0}_{(-1)}] = 0,
$$
and
$$
(\bp\bd)_{(m)} e^{\bq\bd} = 0, \qquad
{\Lambda^i_0}_{(l)} e^{\bq\bd} = \delta_{l,0} q_i e^{\bq\bd}, \qquad m,l\ge0.
$$
Moreover, since $\varepsilon(\delta^i,\delta^j)=1$, we have $e^{\bp\bd} e^{\bq\bd} = e^{(\bp+\bq)\bd}$. From here, we obtain:
\begin{align*}
Y(&{\Lambda^i_0}_{(-1)}e^{\bp\bd},z){\Lambda^j_0}_{(-1)}e^{\bq\bd}\\
&=\Biggl(\sum_{l<0}{\Lambda^i_0}_{(l)}z^{-l-1}\Biggr)\exp\Biggl(\sum_{m>0} (\bp\bd)_{(-m)}\frac{z^m}{m}\Biggr) \bigl({\Lambda^j_0}_{(-1)}-p_jz^{-1}\bigr) e^{(\bp+\bq)\bd}\\
&+\exp\Biggl(\sum_{m>0} (\bp\bd)_{(-m)}\frac{z^m}{m}\Biggr) \bigl(q_i{\Lambda^j_0}_{(-1)}z^{-1}-q_ip_jz^{-2}\bigr) e^{(\bp+\bq)\bd}.
\end{align*}
To find $\bigl({\Lambda^i_0}_{(-1)}e^{\bp\bd}\bigr)_{(n)}\bigl({\Lambda^j_0}_{(-1)}e^{\bq\bd}\bigr)$ for $n \geq 0$, we extract the terms with negative powers of $z$ from the above expression:
$$
\bigl( -p_j{\Lambda^i_0}_{(-1)} z^{-1}+q_i{\Lambda^j_0}_{(-1)} z^{-1}-q_ip_j z^{-2}-q_ip_j(\bp\bd)_{(-1)}z^{-1} \bigr)
e^{(\bp+\bq)\bd} 
$$
This gives us
\begin{align*}
\bigl({\Lambda^i_0}_{(-1)}e^{\bp\bd}\bigr)_{(0)}\bigl({\Lambda^j_0}_{(-1)}e^{\bq\bd}\bigr)&=\bigl(-p_j{\Lambda^i_0}_{(-1)}+q_i{\Lambda^j_0}_{(-1)}-q_ip_j(\bp\bd)_{(-1)}\bigr)e^{(\bp+\bq)\bd}, \\
\bigl({\Lambda^i_0}_{(-1)}e^{\bp\bd}\bigr)_{(1)}\bigl({\Lambda^j_0}_{(-1)}e^{\bq\bd}\bigr)&=-q_ip_je^{(\bp+\bq)\bd},\\
\bigl({\Lambda^i_0}_{(-1)}e^{\bp\bd}\bigr)_{(n)}\bigl({\Lambda^j_0}_{(-1)}e^{\bq\bd}\bigr)&=0 \;\;\text{ for }\; n \geq 2.
\end{align*} 
The other $n$-th products follow from similar reasoning.
\end{proof}

Recall the universal affine vertex algebra $V_k({\fg})$ of level $k$, defined in Section \ref{sec:affva}.
We will consider the tensor product of vertex algebras $V_k({\fg})\otimes V_J$; see \eqref{vert9}.
We can extend $\sigma$ to an automorphism ${\sigma}$ of $V_k({\fg})$ of order $N$ as described in \eqref{aff5}.
Then we extend ${\sigma}$ to an automorphism of the vertex algebra $V_k(\fg)\otimes V_J$, again of order $N$, by letting
\begin{equation}\label{sigma}{\sigma}(a\otimes b)={\sigma} a\otimes b.\end{equation}
Let $M$ be a $\sigma$-twisted $V_k(\fg)$-module, and $M'$ be a $V_J$-module (untwisted). 
Then $\M=M\otimes M'$ is a $\sigma$-twisted $V_k(\fg)\otimes V_J$-module with a state-field correspondence $Y^{\M}$ given by \eqref{vert15}.
Now we can formulate our main theorem, which uses twisted $V_k(\fg)\otimes V_J$-modules to create representations of the twisted toroidal Lie algebra.

\begin{theorem}\label{thm:main}
Let\/ $\sigma$ be an automorphism of order\/ $N$ of a simple finite-dimensional Lie algebra\/ $\fg$, 
and let\/ $\M=M\otimes M'$ be a\/ $\sigma$-twisted\/ $V_k(\fg)\otimes V_J$-module.
Then the Lie algebra of modes of the fields 
\begin{align*}&Y^{\M}(a\otimes e^{\bp\bd}, z), \quad \quad \quad \quad Y^{\M}(\vac\otimes e^{\bp\bd},z),\\
&Y^{\M}(\vac\otimes \delta^i_{(-1)} e^{\bp\bd},z),  \quad \quad Y^{\M}(\vac\otimes {\Lambda^i_0}_{(-1)} e^{\bp\bd},z),
\end{align*} 
where\/ $a \in {\fg}$ and\/ $\bp\in \bZ^r$, form a representation of the twisted toroidal Lie algebra\/ \/ $\hat{\L}_{r+1,k}({\fg},\sigma^{-1})$ of level\/ $k$ on\/ $\M$.
Explicitly, we have a Lie algebra homomorphism\/ $\phi\colon\hat{\L}_{r+1,k}({\fg},\sigma^{-1})\to\End\M$ given by$:$
  \begin{equation}\label{phi}
  \begin{split}
  a\otimes t_0^{m}\bt^\bp&\mapsto \bigl(a \otimes e^{\bp\bd}\bigr)^{\M}_{(\frac{m}{N})}, \\
  K_0\otimes t_0^{Nm}\bt^\bp&\mapsto \frac1N \bigl(\vac\otimes e^{\bp\bd}\bigr)^{\M}_{({m}-1)}, \\
    K_i\otimes t_0^{Nm}\bt^\bp&\mapsto \bigl(\vac\otimes \delta^i_{(-1)}e^{\bp\bd}\bigr)^{\M}_{({m})}, \\
  d_{i}\otimes t_0^{Nm}\bt^\bp&\mapsto \bigl(\vac\otimes {\Lambda^i_0}_{(-1)}e^{\bp\bd}\bigr)^{\M}_{({m})},
  \end{split}
  \end{equation}
for\/ $\bp\in\bZ^r$, $1\le i\le r$, and\/ $a\in\fg$, $m\in \bZ$ such that\/ $\sigma a = e^{-2\pi\ii m/N} a$. 
\end{theorem}
\begin{proof}
Recall that $\L_{r+1}({\fg},\sigma^{-1})_m$ for $m\in\bZ$ is spanned by all elements of the form $a\otimes t_0^{m}\bt^\bp$ where $\bp\in\bZ^r$ and $a\in\fg$ with $\sigma^{-1} a = e^{2\pi\ii m/N} a$.
The latter condition is equivalent to $\sigma a = e^{-2\pi\ii m/N} a$. Hence, $\hat{\L}_{r+1,k}({\fg},\sigma^{-1})$ is spanned by all elements in the left-hand side of \eqref{phi}, subject to the relations:
$$
Nm K_0\otimes t_0^{Nm}\bt^\bp + \sum_{i=1}^r p_i K_i\otimes t_0^{Nm}\bt^\bp = 0.
$$
For the map $\phi$ to be well defined, we need to check that
$$
Nm \frac1N \bigl(\vac\otimes e^{\bp\bd}\bigr)^{\M}_{({m}-1)} + \sum_{i=1}^r p_i \bigl(\vac\otimes \delta^i_{(-1)}e^{\bp\bd}\bigr)^{\M}_{({m})} = 0,
$$
or equivalently,
$$
m\bigl(\vac\otimes e^{\bp\bd}\bigr)^{\M}_{({m}-1)} + \bigl(\vac\otimes (\bp\bd)_{(-1)}e^{\bp\bd}\bigr)^{\M}_{({m})} = 0.
$$
This follows from the translation covariance property \eqref{vert2b}:
\begin{align*}
-m\bigl(\vac\otimes e^{\bp\bd}\bigr)^{\M}_{(m-1)} 
&=\bigl( T(\vac\otimes e^{\bp\bd})\bigr)^{\M}_{(m)} 
=\bigl( \vac\otimes T e^{\bp\bd}\bigr)^{\M}_{(m)} \\
&=\bigl(\vac\otimes (\bp\bd)_{(-1)}e^{\bp\bd}\bigr)^{\M}_{({m})},
\end{align*}
where we used \eqref{Teal} and that $T$ acts as $T\otimes I+I\otimes T$ on a tensor product of vertex algebras. 

To show that $\phi$ is a homomorphism, we need to check that the defining Lie brackets \eqref{2.2}--\eqref{d3} of the twisted toroidal Lie algebra $\hat{\L}_{r+1,k}({\fg},\sigma^{-1})$
match with the commutators of modes in the right-hand side of \eqref{phi}. The latter are determined by the commutator formula \eqref{vert14}.
We can apply \eqref{vert14} because $\sigma a = e^{-2\pi\ii m/N} a$ implies $\sigma (a \otimes e^{\bp\bd}) = e^{-2\pi\ii m/N} (a \otimes e^{\bp\bd})$, since $\sigma$ acts
as the identity on $V_J$.
Similarly, if $b\in\fg$ and $n\in\bZ$ satisfy $\sigma b = e^{-2\pi\ii n/N} a$, then $\sigma (b \otimes e^{\bq\bd}) = e^{-2\pi\ii n/N} (b \otimes e^{\bq\bd})$ for $\bq\in\bZ^r$.
Hence,
\begin{align*}
\Bigl[&(a\otimes e^{\bp\bd})^{\M}_{(\frac mN)}, (b\otimes e^{\bq\bd})^{\M}_{(\frac nN)}\Bigr]
=\sum_{l = 0}^\infty \binom{\frac mN}{l} \bigl((a\otimes e^{\bp\bd})_{(l)}(b\otimes e^{\bq\bd})\bigr)^{\M}_{(\frac mN+\frac nN-l)}.
\end{align*}

The $l$-th product $(a\otimes e^{\bp\bd})_{(l)}(b\otimes e^{\bq\bd})$ is the coefficient of $z^{-l-1}$ in the expression
$$Y(a\otimes e^{\bp\bd},z)(b\otimes e^{\bq\bd})
=Y(a,z)b\otimes Y(e^{\bp\bd},z)e^{\bq\bd}.$$
We are only interested in the negative powers of $z$. By the products \eqref{aff4} in the affine vertex algebra $V_k(\fg)$, we have
$$
Y(a,z)b =  (a|b) k \vac z^{-2} + [a,b] z^{-1} + \cdots,
$$
where $\cdots$ denote terms with higher powers of $z$.
The products in the lattice vertex algebra $V_J$ were computed in Lemma \ref{lem:vjprod}. In particular, we have
$$
Y(e^{\bp\bd},z)e^{\bq\bd} =  e^{(\bp+\bq)\bd} + (\bp\bd)_{(-1)}e^{(\bp+\bq)\bd} z + \cdots.
$$
Putting these together, we get
\begin{align*}
Y(a &\otimes e^{\bp\bd},z)(b\otimes e^{\bq\bd})
=z^{-2} (a|b)k\vac \otimes e^{(\bp+\bq)\bd} \\
&+ z^{-1} [a,b]\otimes e^{(\bp+\bq)\bd}
+z^{-1} (a|b)k\vac \otimes (\bp\bd)_{(-1)}e^{(\bp+\bq)\bd}+\cdots.
\end{align*}
Hence, we obtain the bracket
\begin{align*}
\bigl[&\phi\bigl( a\otimes t_0^{m}\bt^\bp \bigr), \phi\bigl( b\otimes t_0^{n}\bt^\bq \bigr)\bigr] \\
&=\Bigl[\bigl(a\otimes e^{\bp\bd}\bigr)^{\M}_{(\frac mN)}, \bigl(b\otimes e^{\bq\bd}\bigr)^{\M}_{(\frac nN)}\Bigr] \\
&=\bigl([a,b]\otimes e^{(\bp+\bq)\bd}\bigr)^{\M}_{(\frac mN+\frac nN)}
+(a|b)k\bigl(\vac\otimes (\bp\bd)_{(-1)} e^{(\bp+\bq)\bd}\bigr)^{\M}_{(\frac mN+\frac nN)} \\
&\quad+\frac mN (a|b)k\bigl(\vac\otimes e^{(\bp+\bq)\bd}\bigr)^{\M}_{(\frac mN+\frac nN-1)} \\
&=\phi\bigl( [a,b]\otimes t_0^{m+n}\bt^{\bp+\bq} \bigr)
+ k(a|b) \sum_{i=1}^r p_i \phi\bigl(K_i\otimes t_0^{m+n}\bt^{\bp+\bq} \bigr) \\
&\quad+k(a|b) m\phi\bigl(K_0\otimes t_0^{m+n}\bt^{\bp+\bq} \bigr).
 \end{align*}
By \eqref{2.1}, the last expression is exactly $\phi( [a\otimes t_0^{m}\bt^\bp, b\otimes t_0^{n}\bt^\bq] )$.

We compute the other brackets in a similar fashion, and we obtain for
$m,n\in\bZ$, $a\in\bC\vac\oplus\fg$, $1\le i,j\le r$, $\bp,\bq\in\bZ^r$:
\begin{align*}
&\Bigl[\bigl(\vac\otimes u_{(-1)} e^{\bp\bd}\bigr)^{\M}_{(m)}, \bigl(a \otimes e^{\bq\bd}\bigr)^{\M}_{(\frac nN)}\Bigr] \\
&\quad\quad = \Bigl[\bigl(\vac\otimes u_{(-1)} e^{\bp\bd}\bigr)^{\M}_{(m)}, \bigl(\vac\otimes \delta^j_{(-1)} e^{\bq\bd}\bigr)^{\M}_{(n)}\Bigr]= 0, \qquad u\in \{\vac,\delta^i\}, \\
&\Bigl[\bigl(\vac\otimes {\Lambda^i_0}_{(-1)}e^{\bp\bd}\bigr)^{\M}_{(m)}, \bigl(a \otimes e^{\bq\bd}\bigr)^{\M}_{(\frac nN)}\Bigr]=q_i\bigl(a\otimes e^{(\bp+\bq)\bd}\bigr)^{\M}_{(m+\frac nN)}, \\
&\Bigl[\bigl(\vac\otimes {\Lambda_0^i}_{(-1)}e^{\bp\bd}\bigr)^{\M}_{(m)}, \bigl(\vac\otimes \delta^j_{(-1)}e^{\bq\bd}\bigr)^{\M}_{(n)}\Bigr]\\&\quad \quad=q_i\bigl(\vac\otimes \delta^j_{(-1)}e^{(\bp+\bq)\bd} \bigr)^{\M}_{(m+n)}
+\delta_{i,j} \bigl(\vac\otimes (\bp\bd)_{(-1)}e^{(\bp+\bq)\bd} \bigr)^{\M}_{(m+n)} \\
&\quad \quad \quad+m\delta_{i,j} \bigl(\vac\otimes e^{(\bp+\bq)\bd}\bigr)^{\M}_{(m+n-1)},\\
&\Bigl[\bigl(\vac\otimes {\Lambda^i_0}_{(-1)}e^{\bp\bd}\bigr)^{\M}_{(m)}, \bigl(\vac\otimes {\Lambda^j_0}_{(-1)}e^{\bq\bd}\bigr)^{\M}_{(n)}\Bigr]\\
&\quad \quad =q_i\bigl(\vac\otimes {\Lambda^j_0}_{(-1)}e^{(\bp+\bq)\bd}\bigr)^{\M}_{(m+n)}-p_j\bigl(\vac\otimes {\Lambda^i_0}_{(-1)}e^{(\bp+\bq)\bd}\bigr)^{\M}_{(m+n)}\\
&\quad \quad \quad -q_ip_j\bigl(\vac\otimes (\bp\bd)_{(-1)}e^{(\bp+\bq)\bd}\bigr)^{ \M}_{(m+n)}-mq_ip_j\bigl(\vac\otimes e^{(\bp+\bq)\bd}\bigr)^{\M}_{(m+n-1)}.
\end{align*}
A close inspection shows that these brackets agree with the brackets \eqref{2.2}--\eqref{d3} in the twisted toroidal Lie algebra $\hat{\L}_{r+1,k}({\fg},\sigma^{-1})$.
This completes the proof of the theorem.
  \end{proof}

\subsection{Twisted modules over lattice vertex algebras and twisted toroidal Lie algebras}
 
Now let $\fg$ be a simple Lie algebra of type $X_\ell$ where $X=A, D, E$ (simply laced), with a root system $\Delta$ and a root lattice $Q=\bZ\Delta$. 
Let $\sigma$ be an automorphism of the lattice $Q$ of finite order $N$. 

Consider the orthogonal direct sum of lattices
$$
L=Q\oplus J,
$$
where, as before, $J$ is defined by \eqref{J_i}, \eqref{J2}. 
We extend $\sigma$ to an automorphism of $L$, acting as the identity on $J$, and we lift it to an automorphism of the lattice vertex algebra 
$$
V_L \simeq V_Q \otimes V_J.
$$ 
Finally, let
$$ 
\fH = \bC\otimes_\bZ L = \fh\oplus\Span\{ \delta^i,\Lambda_0^i \}_{1\le i\le r},
$$
where $\fh = \bC\otimes_\bZ Q$ is the Cartan subalgebra of $\fg$.

With the above notation, using the Frenkel--Kac construction \cite{fk}, we can reformulate Theorem \ref{thm:main} as follows.

\begin{corollary}\label{last}
For any\/ $\sigma$-twisted\/ $V_L$-module\/ $\M$, the Lie algebra of modes of the fields 
$$
Y^{\M}(e^{\alpha+\bp\bd}, z), \quad  Y^{\M}(h_{(-1)}e^{\bp\bd}, z)
\qquad (\alpha\in\Delta\cup\{0\}, \; h\in\fH, \; \bp\in\bZ^r)
$$
form a representation of the twisted toroidal Lie algebra\/ $\hat{\L}_{r+1,1}({\fg},\sigma^{-1})$ of level\/ $1$ on\/ $\M$. 
\end{corollary}
\begin{proof}
The proof is the same as the proof of Corollary \ref{cor112}.
\end{proof}

In the special case when $\sigma$ is a Coxeter element in the Weyl group of $\fg$, the above corollary recovers Billig's construction from \cite{yb}.

Observe that, for $r=1$, the lattice $Q\oplus\bZ\delta$ is the root lattice of the affine Kac--Moody algebra $\hat\fg$ of type $X_\ell^{(1)}$, 
while $\fH=\fh\oplus\bC\delta\oplus\bC\Lambda_0$ is the Cartan subalgebra of $\hat\fg$.
Moreover, the set $\{\alpha+p\delta \,|\, \alpha\in\Delta\cup\{0\}, \, p\in\bZ\}$ is the union of $\{0\}$ and the root system of $\hat\fg$.
A natural question is whether one can generalize the statement of Corollary \ref{last} to the case when $\sigma$ is an infinite-order
automorphism of $L$, such as, for example, an element of the Weyl group of $\hat\fg$ (cf.\ \cite{BS}).

\subsection*{Acknowledgments}
We are grateful to Kailash Misra and Naihuan Jing for valuable discussions.
The first author was supported in part by a Simons Foundation grant 584741.

\subsection*{Note Added}
After this article was submitted, we learned of the paper \cite{BY} by Billig and Lau, which contains a similar construction as ours
and, furthermore, constructs irreducible modules over the twisted toroidal Lie algebra. We would like to thank Yuly Billig
for bringing the paper \cite{BY} to our attention. 



\begin{thebibliography}{99}

\bibitem{abp}
B. Allison, S. Berman, and A. Pianzola,
\textit{Multiloop algebras, iterated loop algebras and extended affine Lie algebras of nullity $2$}, 
J. Eur. Math. Soc. \textbf{16} (2014), 327--385.

\bibitem{BE}
B. Bakalov and J. Elsinger,
\textit{Orbifolds of lattice vertex algebras under an isometry of order two},
J. Algebra \textbf{441} (2015), 57--83. 

\bibitem{bak} B. Bakalov and V.G. Kac,
\textit{Twisted Modules over Lattice Vertex Algebras},
Lie Theory and its Applications in Physics V, 3-26, World Sci. Publ. River Edge, NJ, 2004.

\bibitem{BS}
B. Bakalov and M. Sullivan, 
\textit{Twisted logarithmic modules of lattice vertex algebras}, 
Trans. Amer. Math. Soc. \textbf{371} (2019), 7995--8027.

\bibitem{bm} S. Berman, Y. Billig, and J. Szmigielski,
\textit{Vertex operator algebras and the representation theory of toroidal algebras. Recent developments in infinite-dimensional Lie algebras and conformal field theory (Charlottesville, VA, 2000)},
1–26, Contemp. Math., 297, Amer. Math. Soc., Providence, RI, 2002. 

\bibitem{sbm} S. Berman, Y. Billig, \textit{Irreducible representations for toroidal Lie algebras,} J. Algebra \textbf{221} (1999), 188–231.

\bibitem{kry} S. Berman, and Y. Krylyuk,
\textit{Universal central extensions of twisted and untwisted Lie algebras extended over commutative rings,} 
J. Algebra,  \textbf{173} (2) (1995), 302--347.

\bibitem{BT}
D. Bernard and J. Thierry-Mieg,
\textit{Level one representations of the simple affine Kac--Moody algebras in
their homogeneous gradations},
Comm. Math. Phys. \textbf{111} (1987), 181--246.

\bibitem{yb}Y. Billig,
\textit{Principal vertex operator representations for toroidal Lie algebras,} 
J. Math. Phys. \textbf{39} (1998), no. 7, 3844--3864. 

\bibitem{BY}
Y. Billig and M. Lau, \textit{Irreducible modules for extended affine Lie algebras,}
J. Algebra \textbf{327} (2011), 208--235.

\bibitem{b}R.E. Borcherds, 
\textit{Vertex algebras, Kac--Moody algebras, and the Monster,} 
Proc. Natl. Acad. Sci. USA \textbf{83}, 3068--3071.

\bibitem{cjkt}
F. Chen, N. Jing, F. Kong, and S. Tan,
\textit{Twisted toroidal Lie algebras and Moody--Rao--Yokonuma presentation}, 
Sci. China Math. (2020).

\bibitem{dong}C. Dong, 
\textit{Twisted modules for vertex algebras associated with even lattices,} 
J. Algebra \textbf{165} (1994), 91--112.

 \bibitem{em}S. Eswara Rao and R.V. Moody, 
\textit{Vertex representations for $n$-toroidal Lie algebras and a generalization of the Virasoro algebras},
Comm. Math. Phys. \textbf{159} (1994), 239--264.
  
\bibitem{fab} M. Fabbri, and R.V. Moody, 
\textit{Irreducible representations of Virasoro-toroidal Lie algebras},
Comm. Math. Phys. \textbf{159} (1994), 1--13.
  
\bibitem{gold} A.J. Feingold, I.B. Frenkel, and J.F.X. Ries,
\textit{Spinor construction of vertex operator algebras, triality, and $E_8^{(1)}$,}
Contemporary Math., 121, Amer. Math. Soc., Providence, RI, 1991.

\bibitem{curves} E. Frenkel and D. Ben-Zvi, 
\textit{Vertex algebras and algebraic curves,}
Math. Surveys and Monographs, 88, Amer. Math. Soc., Providence, RI, 2001; 2nd ed., 2004.

\bibitem{FHL} I.B. Frenkel, Y.-Z. Huang, and J. Lepowsky, 
\textit{On axiomatic approaches to vertex operator algebras and modules,}
 Mem. Amer. Math. Soc. \textbf{104} (1993), no. 494. 
 
\bibitem{fk}I.B. Frenkel and V.G. Kac,
\textit{Basic representations of affine Lie algebras and dual resonance models},
Invent. Math. \textbf{62} (1980), 23--66.

\bibitem{flm} I.B. Frenkel, J. Lepowsky, and A. Meurman, 
\textit{Vertex Operator Algebras and the Monster},
Pure and Appl. Math., 134. Academic Press, Boston (1988).

\bibitem{fz} I.B. Frenkel and Y. Zhu,  
\textit{Vertex operator algebras associated to representations of affine and Virasoro algebras,}
 Duke Math. J. \textbf{66} (1992), no. 1, 123--168. 

\bibitem{FJ} J. Fu and C. Jiang, 
\textit{Integrable representations for the twisted full toroidal Lie algebras}, 
J. Algebra, \textbf{307} (2007), 769--794.

\bibitem{GNOS}
P. Goddard, W. Nahm, D. Olive, and A. Schwimmer, 
\textit{Vertex operators for non-simply-laced algebras},
Comm. Math. Phys. \textbf{107} (1986), 179--212.

\bibitem{misra} N. Jing, C.R. Mangum, and K.C. Misra, 
\textit{On Realization of some Twisted Toroidal Lie Algebras,} 
Lie algebras, vertex operator algebras, and related topics, 139--148.
Contemp. Math., Amer. Math. Soc., Providence, RI, 2017.

 \bibitem{jing}N. Jing and K.C. Misra, 
  \textit{Fermionic realization of toroidal Lie algebras of classical types,}
J. Algebra, \textbf{324} (2010), 183--194.

 \bibitem{jing2}N. Jing, K.C. Misra, and C. Xu, 
\textit{Bosonic realization of toroidal Lie algebras of classical types,}  
Proc. Amer. Math. Soc. \textbf{137} (2009), no. 11, 3609--3618. 

\bibitem{IDLA} V.G. Kac, 
\textit{Infinite-dimensional Lie algebras},
Third edition. Cambridge University Press, Cambridge, 1990.

\bibitem{VAFB} V.G. Kac, 
\textit{Vertex algebras for beginners},
Second edition. University Lecture Series, vol 10, American Mathematical Society, 1998. 

\bibitem{KKLW}V.G. Kac, D.A. Kazhdan, J. Lepowsky, and R.L. Wilson, 
\textit{Realization of the basic representations of the Euclidean Lie algebras,}
Adv. in Math. \textbf{42} (1981), 83--112.

\bibitem{112} V.G. Kac and D.H. Peterson, 
\textit{112 Constructions of the basic representation of the loop group of E8},
Symposium on anomalies, geometry, topology (Chicago, Ill., 1985), 276-298, World Sci. Publishing, Singapore, 1985.

\bibitem{bombay} V.G. Kac, A.K. Raina, and N. Rozhkovskaya, 
\textit{Bombay lectures on highest weight representations of infinite dimensional Lie algebras},
Second edition. Advanced Series in Mathematical Physics, 29. World Scientific Publishing Co. Pte. Ltd., Hackensack, NJ, 2013. 

 \bibitem{KT}V.G. Kac and I.T. Todorov, 
 \textit{Affine orbifolds and rational conformal field theory
extensions of\/ $W_{1+\infty}$}, Comm. Math. Phys. \textbf{190} (1997), 57--111.

\bibitem{gannon} T. Gannon, 
\textit{Moonshine beyond the Monster. The bridge connecting algebra, modular forms and physics}, 
Cambridge Monographs on Math. Phys., Cambridge Univ. Press, Cambridge, 2006.

 \bibitem{kassel}C. Kassel, 
 \textit{ Kähler differentials and coverings of complex simple Lie algebras extended over a commutative algebra},
 Proceedings of the Luminy conference on algebraic K-theory (Luminy, 1983). J. Pure Appl. Algebra \textbf{34} (1984), no. 2-3, 265--275.
 
  \bibitem{kroode}F. ten Kroode, J. van de Leur, 
 \textit{Level-one representations of the affine Lie algebra $B_{n}^{(1)}$},
 Acta Appl. Math. \textbf{31} (1993), no.1, 1--73.
 
  \bibitem{lep1}J. Lepowsky, 
   \textit{Calculus of twisted vertex operators},
 Proc. Nat. Acad. Sci. USA \textbf{82} (1985), 8295--8299.

 \bibitem{li}J. Lepowsky and H. Li, 
 \textit{Introduction to vertex operator algebras and their representations,}
 Progress in Math., 227, Birkh\"{a}user Boston, Boston, MA, 2004.

 \bibitem{lep2}J. Lepowsky and R.L. Wilson, 
  \textit{The structure of standard modules. I. Universal algebras and the Rogers-Ramanujan identities},
Invent. Math. \textbf{77} (1984), no. 2, 199--290. 

 \bibitem{lep3} J. Lepowsky and R.L. Wilson,
 \textit{Construction of the affine Lie algebra $A_1^{(1)}$},
Comm. Math. Phys. \textbf{62} (1978), 43--53.
 
 
 \bibitem{misra1}K. C. Misra,
 \textit{Level one standard modules for affine symplectic Lie algebras}, 
Math. Ann. \textbf{287} (1990), no. 2, 287–302. 

 \bibitem{misra2}K. C. Misra,
 \textit{Realization of the level one standard $\tilde{C}_{2k+1}$-modules},
 Trans. Amer. Math. Soc. \textbf{321} (1990), no. 2, 483–504.
  
\bibitem{mry}R.V. Moody, S.E. Rao, and T. Yokonuma, 
\textit{Toroidal Lie algebras and vertex representations},
 Geom. Dedicata \textbf{35} (1990), no. 1-3, 283--307.
 
 \bibitem{tan1} S. Tan, 
 \textit{Principal construction of the toroidal Lie algebra of type $A_1$},
Math. Zeit. \textbf{230} (1999), 621--657.

 \bibitem{tan2} S. Tan, 
\textit{Vertex operator representations for toroidal Lie algebras of type $B_l$},
Comm. Algebra \textbf{27} (1999), 3593--3618.

 \bibitem{xj}Y. C. Xu  and C. P. Jiang, \textit{Vertex operators of $G_2^{(1)}$ and $B_l^{(1)}$}, J. Phys. A \textbf{23} (1990), no. 14, 3105--3121.
\end{thebibliography}
\end{document}